\documentclass[11 pt]{amsart}
\usepackage{graphicx}

\usepackage{amsthm,amsmath,latexsym,amssymb,amsfonts,amscd,hyperref}
\usepackage{color,tikz,tikz-cd}
\usepackage[all]{xy}
\newtheorem{theorem}{Theorem}[section]

\newtheorem{proposition}[theorem]{Proposition}

\theoremstyle{definition}
\newtheorem{definition}[theorem]{Definition}

\theoremstyle{remark}
\newtheorem{remark}[theorem]{Remark}

\numberwithin{equation}{section}

\begin{document}
\phantom{a}
\vspace{-1.5cm}
\title[ Cup product on $A_\infty$-cohomology and deformations ]{Cup product on $A_\infty$-cohomology and deformations}  

\author{ Alexey A. Sharapov }
\address{Physics Faculty, Tomsk State University, Lenin ave. 36, Tomsk 634050, Russia}
\email{sharapov@phys.tsu.ru}
\thanks {The work of the first author was supported by the Ministry of Science and Higher Education of the Russian Federation, Project No. 0721-2020-0033.}

\author{Evgeny D. Skvortsov}
\address{Albert Einstein Institute, 
Am M\"{u}hlenberg 1, D-14476, Potsdam-Golm, Germany}
\address{Lebedev Institute of Physics, 
Leninsky ave. 53, 119991 Moscow, Russia}
\email{evgeny.skvortsov@aei.mpg.de}
\thanks{The work of the second author was supported by the Russian Science Foundation grant 18-72-10123 in association with the Lebedev Physical Institute.  }




\begin{abstract}
We propose a simple method for constructing formal deformations of differential graded algebras  in the category of minimal $A_\infty$-algebras. 
The basis for our approach is provided by the Gerstenhaber algebra structure on $A_\infty$-cohomology, which we define in terms of the brace operations. As an example, we construct a minimal $A_\infty$-algebra from the Weyl-Moyal $\ast$-product algebra of polynomial functions. 
\end{abstract}

\maketitle
\section{Introduction}
The concept of homotopy associative algebras (or $A_\infty$-algebras), which first appeared  in the context of algebraic topology \cite{St}, has now evolved into a mature algebraic theory with numerous applications in theoretical and mathematical physics \cite{Kontsevich:2006jb, JSt}. String field theory \cite{KS,Erler:2013xta}, the deformation quantization of gauge systems \cite{LSh},  non-commutative field theory \cite{BBKL},  and higher-spin gravity \cite{ShSk1, LZ, ShSk3} are just a few examples where these algebras play a dominant role.  It turns out that many of $A_\infty$-algebras encountered  in applications are obtained by deforming differential graded algebras (DGA) or their families.  The general deformation problem for $A_\infty$-algebras has been considered in Refs. \cite{PSh, FP, Tor, ShSk2}. 

In this paper, we propose a simple formula for the  deformation of families of DGA's in the category of minimal $A_\infty$-algebras. The basis for our construction is provided by a cup product on $A_\infty$-cohomology. As was first shown by Getzler \cite{Getz}, each $A_\infty$-structure $m\in\mathrm{Hom}(T(V),V)$ on a graded vector space $V$ gives rise to an $A_\infty$-structure $M$ on the vector space $\mathrm{Hom}(T(V),V)$. 
As with any $A_\infty$-algebra, the second structure map $M_2$ induces a multiplication operation, called cup product, in the $A_\infty$-cohomology defined by the differential $M_1$  and we use this operation to deform the original family of $A_\infty$-structures $m$.

The main results of our paper can be summarized in the following 

\begin{theorem}\label{Th1}
Given a one-parameter family   $A=\bigoplus A_n$ of DGA's with differential $\partial: A_n\rightarrow A_{n-1}$, one can define a minimal $A_\infty$-algebra deforming the associative product in $A$ in  the direction of an (inhomogeneous) Hochschild cocycle $\Delta$ given by any linear combination of 
$$
\Delta_n (a_1, a_2, \ldots, a_n)=(a_1 \cdot a_2)'\cdot \partial a_3\cdots \partial a_n \,, \qquad \forall a_i\in A\,.
$$
Here $[\Delta_n]\in HH^n(A,A)$ and the prime stands for the derivative of the dot product in $A$ w.r.t. the parameter.  Each solution to the equation $a\cdot a=0$ for $a\in A_1$ can be deformed to a Maurer--Cartan element of the $A_\infty$-algebra above. 

\end{theorem}

The theorem above admits various interesting specializations, of which we mention only one. Let  $\mathcal{M}$ be a one-parameter family of bimodules over $\mathcal{A}$ . Then one  can define the family of graded algebras  $A=A_0\bigoplus A_1$, where $A_0=\mathcal{A}$, $A_1=\mathcal{M}$, and the product is given by 
\begin{equation}\label{Ext}
\begin{array}{c}
(a_1,m_1)(a_2,m_2)=(a_1a_2, a_1m_2+m_1a_2)\\[3mm]
\forall a_1,a_2\in \mathcal{A}\,, \quad\forall m_1,m_2\in \mathcal{M}\,. 
\end{array}
\end{equation}
This is known as the trivial extension of the algebra $\mathcal{A}$ by the bimodule $\mathcal{M}$. In order to endow the algebra $A$ with a differential $\partial$, we consider the $\mathcal{A}$-dual bimodule  
$$
\mathcal{M}^\ast=\mathrm{Hom}_{\mathcal{A}-\mathcal{A}}(\mathcal{M},\mathcal{A})\,.
$$
Each element $h\in \mathcal{M}^\ast$ extends to an  $\mathcal{A}$-bimodule homomorphism $\tilde{h}: A\rightarrow \mathcal{A}$ by setting $\tilde{h}(a)=0$, $\forall a\in \mathcal{A}$.  
In case $\ker h=0$, one can easily see that $\tilde{h}$ is a derivation of  the algebra $A$ of degree $-1$. Furthermore, it follows from the definition that $\tilde{h}^2=0$. Hence, we can put $\partial =\tilde{h}$. The  deformation of the algebra $A=A_0\bigoplus A_1$ stated by  Theorem \ref{Th1}  yields then a deformation of the $\mathcal{A}$-bimodule $\mathcal{M}$ in the category of  minimal $A_\infty$-algebras. 

Notice that in the above construction $h(\mathcal{M})$ is a two-sided ideal in $\mathcal{A}$. Conversely, given a two-sided ideal $\mathcal{I}\subset \mathcal{A}$, we can set $\mathcal{M}=\mathcal{I}$ and take $h$ to be the inclusion map $\mathcal{M}  \hookrightarrow \mathcal{A}$. This allows one to canonically associate an $A_\infty$-algebra to any pair $(\mathcal{I},\mathcal{A})$. In the particular case $\mathcal{I}=\mathcal{A}$, we get a deformation of the family $\mathcal{A}$ itself. 
For this reason it is natural to term these and the other deformations following from  Theorem \ref{Th1} the {\it inner deformations of families}. 

The rest of the paper is organized as follows.  In Sec. 2,  we review some background material on $A_\infty$-algebras and braces. In Sec. 3,  we define the cohomology groups associated  to an $A_\infty$-structure and endow them with a commutative and associative cup product.  This product operation is then used in Sec. 4 for constructing inner deformations of multi-parameter families of $A_\infty$-algebras and, in particular, DGA's.  Here we also introduce the concept of  local finiteness for families and show that each inner deformation of a locally finite family of $A_\infty$-algebras induces a deformation of the corresponding Maurer--Cartan elements. By way of illustration, we finally construct a minimal $A_\infty$-algebra that deforms the algebra of polynomial functions regarded as a bimodule over itself. The deformation is completely determined by a canonical Poisson bracket and can be viewed as a certain generalization of the Weyl--Moyal $\ast$-product.

\section{$A_\infty$-algebras and braces}

Throughout the paper we  work over a fixed ground field  $k$ of characteristic zero. All tensor products and $\mathrm{Hom}$'s are defined over $k$ unless otherwise indicated.
 We begin by recalling  some basic definitions and
constructions related to  $A_\infty$-algebras.

Let $V=\bigoplus V^l$ be a $\mathbb{Z}$-graded vector space over $k$ and let
$T(V)=\bigoplus_{n\geq 0} V^{\otimes n}$ denote its  tensor algebra;
it is understood that $T^0(V)=k$.  The $k$-vector spaces $T(V)$ and
$\mathrm{Hom}(T(V),V)$ naturally inherit the grading of $V$. The vector  space
\begin{equation}\label{Hom}\mathrm{Hom}(T(V),V)=\bigoplus_l \mathrm{Hom}^l(T(V),V)\end{equation} 
is known to carry the structure of a graded Lie algebra. This is defined as follows. For any
two homogeneous homomorphisms  $f\in \mathrm{Hom}(T^n(V),V)$
and $g\in \mathrm{Hom} (T^{m}(V),V)$, one first defines a (non-associative) {\it composition } product \cite{GSch} as 
\begin{equation}\label{comp}
(f\circ g)(v_1\otimes v_2\otimes\cdots\otimes v_{m+n-1} )
\end{equation}
$$ =\sum_{i=0}^{n-1}(-1)^{|g|\sum_{j=1}^i|v_j|} f(v_1\otimes \cdots\otimes v_i\otimes g(v_{i+1}\otimes \cdots\otimes v_{i+m})\otimes \cdots \otimes v_{m+n-1}).
$$
Here $|g|$ denotes the degree of $g$ as a linear map of
graded vector spaces\footnote{We define 
the degree of multi-linear maps as in  \cite{KS}. A more conventional
$\mathbb{Z}$-grading \cite{Gerst}, \cite{Tsygan} on $\mathrm{Hom}(T(V),V)$ is
related to ours by {\it suspension}:  $V\rightarrow V[-1]$, where 
$V[-1]^l=V^{l-1}$.}. 
Then the graded Lie bracket on (\ref{Hom}) is given by the {\it Gerstenhaber bracket} \cite{Gerst}
\begin{equation}\label{GB}
[f,g]=f\circ g-(-1)^{|f||g|}g\circ f\,. 
\end{equation}
One can see that the Gerstenhaber bracket is graded skew-symmetric,
$$
[f,g]=-(-1)^{|f||g|}[g,f]\,,
$$
and obeys the graded Jacobi identity
$$
[[f,g],h]=[f,[g,h]]-(-1)^{|f||g|}[g,[f,h]]\,.
$$
In particular, $[f,f]=2f\circ f$ for any odd $f$. 

\begin{definition}
 An {$A_{\infty}$-structure} on a $\mathbb{Z}$-graded vector space $V$ is given by an
element $m\in \mathrm{Hom}^1(T(V),V)$ obeying the Maurer--Cartan (MC) equation
\begin{equation}\label{MC}
m\circ m=0\,.
\end{equation}
The pair $(V, m)$ is called the {$A_\infty$-algebra}.
\end{definition}

By definition, each  $A_\infty$-structure $m$ is given by an (infinite) sum $m=m_0+m_1+m_2+\ldots$ of
multi-linear maps $m_n\in \mathrm{Hom}(T^n(V),V)$. Expanding (\ref{MC}) into homogeneous components yields  an infinite collection of quadratic relations on the $m_n$'s, which are known as the Stasheff identities
\cite{St}. An $A_{\infty}$-algebra is called {\it flat} if $m_0=0$.
For flat algebras, the first structure map $m_1:
V^{l}\rightarrow V^{l+1}$ squares to zero, $m_1\circ m_1=m_1^2=0$; hence, it 
makes  $V$ into a cochain complex.  An
$A_\infty$-algebra is called {\it minimal} if $m_0=m_1=0$. In the minimal case, the second structure map  $m_2: V\otimes V\rightarrow V$ endows the space $V[-1]$ with the structure of a graded associative algebra w.r.t. the dot product
\begin{equation}\label{bbb}
u\cdot v=(-1)^{|u|-1}m_2(u\otimes v)\,,
\end{equation}
associativity being provided by the Stasheff identity
$m_2\circ m_2=0$. This allows one to regard  a graded associative algebra
as a `very degenerate'  $A_\infty$-algebra with $m=m_2$.  More generally, an
$A_\infty$-structure $m=m_1+m_2$ gives rise to a  DGA  $(V[-1],d, \cdot )$ with the product (\ref{bbb}) and
the differential $d=m_1$. The graded Leibniz rule
$$
d(u\cdot v)=du\cdot v+(-1)^{|u|-1}u\cdot dv
$$
follows from the Stasheff identity $[m_1,m_2]=0$.

The composition product (\ref{comp}) is a representative of the infinite sequence of multi-linear operations on  $\mathrm{Hom}(T(V),V)$ known as {\it braces}. The braces first appeared in the work of Kadeishvili \cite{Kad} and were then studied by several authors \cite{Getz, GV, GJ}. To  simplify subsequent formulas, let us denote $W=\mathrm{Hom}(T(V),V)$. 

\begin{definition}
 Given homogeneous elements  $A, A_1,\ldots, A_m\in W$ and  $v_1,\ldots,v_n\in V$, define the braces $A\{A_1,\ldots, A_m\}\in W$, $m=0,1,2,\ldots$,  by the formula
\begin{equation}
\begin{array}{c}
  A\{A_1,\ldots, A_m\}(v_1,\ldots,v_n) \\[5mm] \displaystyle
    =\sum_{0\leq k_1\leq \cdots\leq k_m\leq n} (-1)^\epsilon A(v_1,\ldots, v_{k_1}, A_1(v_{k_1+1}, \ldots), \\[5mm]  \hspace{3cm}\ldots,v_{k_m}, A_{m}(v_{k_m+1},\ldots),\ldots,v_n)\,,
\end{array}
\end{equation}
where $\epsilon =\sum_{i=1}^m|A_i|\sum_{j=1}^{k_i}|v_j|$. It is assumed that $A\{\varnothing\}=A$. 
\end{definition}
It follows from the definition that 
\begin{equation}\label{circ}
      A\{A_1\}=A\circ A_1\,.
\end{equation}

The braces obey the so-called higher pre-Jacobi identities \cite{GV}
\begin{equation}\label{pre-J}
\begin{array}{c}
    A\{A_1,\ldots, A_m\}\{B_1,\ldots, B_n\}\\[5mm]
\displaystyle =\sum_{\mbox{\small $AB$-shuffles}} (-1)^\epsilon A\{B_1,\ldots, B_{k_1},A_1\{B_{k_1+1}, \ldots\},\\[7mm]
\hspace{2,5cm}\ldots, B_{k_m}, A_m\{B_{k_m+1},\ldots\},\ldots, B_n\}\,,
\end{array}
\end{equation}
where $\epsilon=\sum_{i=1}^m|A_i|\sum_{j=1}^{k_i}|B_j|$.
Here summation is over all shuffles of the $A$'s and $B$'s (i.e., 
the order of elements in either group is preserved under permutations) and
the case of empty braces  $A_k\{\varnothing\}$ is not excluded.

In \cite{Getz}, Getzler have shown that any $A_\infty$-structure $m$ on $V$ can be lifted to a flat $A_\infty$-structure $M$ on $W$ by setting
\begin{equation}\label{M}
\begin{array}{l}
   M_0(\varnothing)=0\,,\\[5mm]
   M_1(A)=m\circ A-(-1)^{|A|}A\circ m \,,\\[5mm]
M_k(A_1,\ldots,A_k)=m\{A_1,\ldots,A_k\}\,,\qquad k>1\,.
   \end{array}
\end{equation}
Indeed, by the definition of the composition product (\ref{comp})
\begin{equation}
    (M\circ M)(A_1,\ldots, A_n)
\end{equation}
$$
=\sum_{0\leq i
    \leq j\leq n}(-1)^{\varepsilon} M(A_1,\ldots, A_{i-1}, M(A_i,\ldots,A_j), A_{j+1},\ldots, A_n)\,,
$$
where $\varepsilon=\sum_{j=1}^{i-1}|A_j|$. This gives
\begin{equation}\label{MM}
\begin{array}{c}
(M\circ M)(A_1,\ldots, A_n)\\[7mm]=\displaystyle \sum_{0\leq i
    \leq j\leq n}(-1)^{\varepsilon} m\{A_1,\ldots, A_{i-1}, m\{A_i,\ldots,A_j\}, A_{j+1},\ldots, A_n\}
\\[7mm]
-\displaystyle \sum_{1\leq i
    \leq n}(-1)^{\varepsilon} m\{A_1,\ldots, A_{i-1}, A_i\circ m, A_{i+1},\ldots, A_n\}
\\[7mm]
+\displaystyle (-1)^{\sum_{i=1}^n|A_i|} m\{A_1,\ldots,A_n\}\circ m\,.
\end{array}
\end{equation}
Using the pre-Jacobi identities (\ref{pre-J}), one can rewrite the last term as 
$$
(-1)^{\sum_{i=1}^n|A_i|} m\{A_1,\ldots,A_n\}\circ m
$$
$$
=\sum_{1\leq i
    \leq n}(-1)^{\varepsilon} \Big(m\{A_1,\ldots, A_{i-1}, A_i\circ m, A_{i+1},\ldots, A_n\}
$$
$$
+  m\{A_1,\ldots, A_i, m, A_{i+1},\ldots, A_n\}\Big)\,.
$$
Then the r.h.s. of Eq. (\ref{MM}) takes the form of 
$$
    (m\circ m)\{A_1,\ldots, A_n\}=m\{m\}\{A_1,\ldots, A_n\}
$$
$$
=\sum_{\mbox{\small $mA$-shuffles}}(-1)^{\varepsilon} m\{A_1,\ldots, A_{i-1}, m\{A_i,\ldots,A_j\}, A_{j+1},\ldots, A_n\}\,.
$$
Hence, 
$$
(M\circ M)(A_1,\ldots,A_n)=(m\circ m)\{A_1,\ldots,A_n\}=0\,.
$$

In what follows we will refer to (\ref{M}) as the {\it derived} $A_\infty$-structure.\footnote{Do not confuse with the derived $A_\infty$-algebras in the sense of Sagave \cite{S}. }

\section{$A_\infty$-cohomology}

If $(V,m)$ is an $A_\infty$-algebra, then the first map 
of the derived $A_\infty$-structure (\ref{M}) makes the graded vector space $W=\bigoplus W^n$ into a cochain complex w.r.t. the differential $M_1: W^n\rightarrow W^{n+1}$. Let ${H}^n(W)$ denote the corresponding cohomology groups.\footnote{For an associative algebra $A$, this cohomology is simply the Hochschild cohomology of the algebra. In that case one normally uses a more standard notation $HH^{n+1}(A,A)$ for the groups $H^{n}
(W)$.} Following \cite{PSh}, we refer to them as $A_\infty$-cohomology groups. The most interesting for us are the groups ${H}^1(W)$ and ${H}^2(W)$, which control the formal deformations of the underlying $A_\infty$-structure $m$. Let us give some relevant definitions.  

When dealing with formal deformations of algebras, one first extends the ground field $k$ to the algebra $k[[t]]$, with the formal variable $t$ playing the role of  a {\it deformation parameter}.  Since $k[[t]]$ is commutative, the graded Lie algebra structure on $W$ extends naturally to $W\otimes k[[t]]$ and then to its completion $\mathcal{W}=W\hat{\otimes} k[[t]]$ w.r.t. the $t$-adic topology. By definition, the  elements of $\mathcal{W}$ are given by the formal power series 
\begin{equation}\label{m-t}
m_t=m^{_{(0)}}+m^{_{(1)}}t+m^{_{(2)}}t^2+\cdots\,,\qquad m^{_{(i)}}\in W\,.
\end{equation}
The natural augmentation $\varepsilon: k[[t]]\rightarrow k$ induces the $k$-homomorphism $\pi: \mathcal{W}\rightarrow W$, which sends the deformation parameter to zero. We say that an MC element $m_t\in \mathcal{W}^1$ is a deformation of  $m\in W^1$ if $\pi(m_t)=m$ or, what is the same, $m_{(0)}=m$ in (\ref{m-t}). Extending now the homomorphism $m_t$ from $V$ to $\mathcal{V}=V\otimes k[[t]]$ by $k[[t]]$-linearity and $t$-adic continuity, we get an $A_\infty$-algebra $(\mathcal{V},m_t)$ which is referred to as the deformation of the algebra $(V,m)$. 
The element $m_{(1)}$ in (\ref{m-t}) is called the {\it first-order deformation} of $m$. 

Two formal deformations $m_t$ and $\tilde{m}_t$ of one and the same $A_\infty$-structure $m$ are considered as { equivalent} if there exists an element $w\in \mathcal{W}^0$ such that $$
\tilde{m}_t=e^{tw}m_t e^{-tw}=\sum_{n=0}^\infty\frac{t^n}{n!}(\mathrm{ad}_w)^n (m_t)\,.
$$
This induces an equivalence relation on the space of first-order deformations and it is the standard fact of 
algebraic deformation theory (see e.g. \cite{PSh}) that the space of nonequivalent first-order deformations is isomorphic to  $H^1(W)$.   If in addition $H^2(W)=0$, then each first-order deformation extends to all orders. 

Since the differential $M_1$ is, by definition, an inner derivation of the graded Lie algebra $W$, the Gerstenhaber bracket (\ref{GB}) induces a Lie bracket on the cohomology space $H^{\bullet}(W)$, for which we use the same bracket notation.  The graded Lie algebra structure on $H^{\bullet}(W)$ can further be extended to the structure of a graded Poisson (or Gerstenhaber) algebra w.r.t. a cup product. The latter is defined as follows. 

By definition, the second structure map $M_2: W\otimes W\rightarrow W$ of the derived $A_\infty$-algebra obeys the identity 
\begin{equation}\label{M12}
M_1(M_2(A,B))+M_2(M_1(A),B)+(-1)^{|A|}M_2(A,M_1(B))=0
\end{equation}
for all $A,B\in W$. From this relation we conclude that (i) $M_2(A,B)\in W$ is an $M_1$-cocycle whenever $A$ and $B$ are so and (ii) the cocycle $M_2(A,B)$ is trivial whenever one of the cocycles $A$ and $B$ is an $M_1$-coboundary. To put this another way,  the map $M_2$ descends to the cohomology inducing a homomorphism    
$$
M_2^\ast: H^n(W)\otimes H^m(W)\rightarrow H^{n+m+1}(W)\,.
$$
We can interpret this homomorphism as a multiplication operation making the suspended vector space $H^{\bullet-1}(W)$ into a $\mathbb{Z}$-graded algebra. More precisely, we set  
\begin{equation}\label{cup}
a \cup b= (-1)^{|A|-1}M_2(A,B)\,,
\end{equation}
where $A, B\in W$ are cocycles representing the cohomology classes $a, b\in H^{\bullet-1}(W)$.
The properties of the cup product are described by the following proposition. 

\begin{proposition}
The cup product (\ref{cup}) endows the  space $H^{\bullet-1}(W)$ with the structure of an associative and graded commutative algebra.
\end{proposition}

\begin{proof}
Associativity follows immediately from the Stasheff identity 
$$
M_2\circ M_2=-[M_1,M_3]\,.
$$
The r.h.s. obviously vanishes when evaluated on $M_1$-cocycles modulo coboundaries, while the l.h.s. takes the form
of the associativity condition
$$
(a\cup b)\cup c-a\cup (b\cup c)=0\,.
$$

The proof of graded commutativity is a bit more cumbersome.  Consider the cochain 
$$
D(A,B)= M_1(A\circ B)-M_1(A)\circ B-(-1)^{|A|} A\circ M_1(B)\,,
$$
which measures  the deviation of $M_1$ from being a derivation of the composition product.  Using the definitions (\ref{circ}) and (\ref{M}), we can write 
\begin{equation}\label{D}
\begin{array}{rcl}
D(A,B)&=& m\{A\{B\}\}-(-1)^{|A|+|B|}A\{B\}\{m\}\\[5mm]
&-&m\{A\}\{B\}+(-1)^{|A|}A\{m\}\{B\}\\[5mm]
&-&(-1)^{|A|}A\{m\{B\}\}+(-1)^{|A|+|B|}A\{B\{m\}\}\,.
\end{array}
\end{equation}
It follows from the  pre-Jacobi identities  (\ref{pre-J}) that  
$$
m\{A\}\{B\}= m\{A,B\}+m\{A\{B\}\}+(-1)^{|A||B|}m\{B,A\}\,.
$$
Applying similar transformations to the other terms in (\ref{D}), we find that all but two terms cancel leaving  
$$
D(A,B)=-M_2(A,B)-(-1)^{|A||B|}M_2(B,A)\,.
$$
It remains to note that for any pair of cocycles $A$ and $B$ the cochain $D(A,B)$ is a coboundary, whence 
$$a\cup b=(-1)^{(|a|-1)(|b|-1)}b\cup a\,.$$
\end{proof}

Notice that for graded associative algebras 
the associativity of the cup product (\ref{cup}) takes place at the level of cochains. This product, however, may not be graded commutative until passing to the Hochschild cohomology.  
\begin{proposition}\label{PR}
The cup product and the Gerstenhaber bracket satisfy the graded Poisson relation
$$
[a,b\cup c] =[a,b]\cup c+(-1)^{|a|(|b|+1)} b\cup [a,c]\qquad \forall a,b,c\in H^\bullet(W)\,.
$$
\end{proposition}
\begin{proof}
The Poisson relation follows from the identity 
$$
\begin{array}{c}
[A,M_2(B, C)]-(-1)^{|A|}M_2([A,B], C) -(-1)^{|A|(|B|+1)}M_2(B, [A,C])\\[5mm]
=(-1)^{|A|} \Big( M_1(A\{B,C\}) - M_1(A)\{B,C\}\\[5mm]
-(-1)^{|A|}A\{M_1(B),C\}-(-1)^{|A|+|B|}A\{B,M_1(C)\}\Big)\,,
\end{array}
$$
which holds for all $A,B,C\in W$. One can verify it directly by making use of the pre-Jacobi identities (\ref{pre-J}). 
\end{proof}

Thus, the cup product and the Gerstenhaber bracket define the structure of a graded Poisson algebra on the $A_\infty$-cohomology $H^\bullet(W)$.

\begin{remark}
The structure of a graded Poisson algebra on the Hochschild cohomology $HH^\bullet(A,A)$ of an associative algebra $A$ was first observed by Gerstenhaber \cite{Gerst}. One can view the  two propositions above as a straightforward extension of  Gerstenhaber's results to the case of $A_\infty$-algebras.  
\end{remark}

\section{Inner deformations of families}

\subsection{Families of algebras} Let $\mathcal{A}_t$ be an $n$-parameter, formal  deformation of an $A_\infty$-algebra $\mathcal{A}$, i.e., the $A_\infty$-structure on $\mathcal{A}_t$ is given by an element $m\in\mathcal{W}= W[[t_1,\ldots, t_n]]$ such that $m\circ m=0$ and $m|_{t=0}$ gives the products in $\mathcal{A}$. Here we allow the deformation parameters $t_i$ to have  non-zero $\mathbb{Z}$-degrees contributing to the total degree $|m|=1$ of $m$ as an element of $\mathcal{W}^1$.  For the sake of simplicity, however, we restrict ourselves to the case where all the degrees $|t_i|$ are {\it even}. Extension to the general case is straightforward (see Remark \ref{VF} below). In the following we will refer to $\mathcal{A}_t$ as a family of $A_\infty$-algebras.  

Let us denote 
$$
m_{(i_1i_2\cdots i_k)}=\frac{\partial^k m}{\partial t_{i_1}\partial t_{i_2}\cdots \partial t_{i_k}}\in \mathcal{W}\,.
$$
Clearly, $|m_{(i_1i_2\cdots i_k)}|={1 -|t_{i_1}|-\ldots-|t_{i_k}|}$.  Taking the partial derivative of the defining relation $m\circ m=0$ w.r.t. the parameter $t_{i}$, we get 
$$[m, m_{(i)}]=M_1(m_{(i)})=0\,.$$ 
In other words, the cochain $m_{(i)}$ is a cocycle of the differential $M_1$ associated to the $A_\infty$-structure $m$. So, $m_{(i)}$ defines a cohomology class of $H^{1-|t_i|}(\mathcal{W})$.

Denote by $\mathcal{D}_m$ the subalgebra in the graded Poisson algebra $H^{\bullet}(\mathcal{W})$ generated by the cocycles $m_{(i)}$. 
\begin{proposition}
The Gerstenhaber bracket on $\mathcal{W}$ induces the trivial Lie bracket on $\mathcal{D}_m$. 
\end{proposition}
\begin{proof}
Differentiating the relation $m\circ m=0$ twice, we get 
$$
[m, m_{(i,j)}]=-[m_{(i)}, m_{(j)}]\,.
$$
Hence, the bracket $[m_{(i)}, m_{(j)}]$ is an $M_1$-coboundary. By Proposition  (\ref{PR}) this result is extended to arbitrary cup products of $m_{(i)}$'s.

\end{proof}

\subsection{Inner deformations} We see that the algebra $\mathcal{D}_m$ is generated by the cup products of the partial derivatives $m_{(i)}$, so that the elements of $\mathcal{D}_m$ are represented by cup polynomials\footnote{By abuse of notation, we write the cup product for the cocycles rather than their cohomology classes.}
\begin{equation}\label{pol}
\Delta=\sum_{l=0}^L c^{i_1\cdots i_l}m_{(i_1)}\cup m_{(i_2)}\cup\cdots \cup m_{(i_l)}\,,
\end{equation}
where $c^{i_1\cdots i_l}\in k[[t_1,\ldots,t_n]]$. Note that with our restriction on the degrees of $t$'s the graded associative algebra $\mathcal{D}_m=\bigoplus \mathcal{D}_m^l$ is purely commutative as it consists only of even elements.   
\begin{proposition}\label{def}
Let $m\in \mathcal{W}^1$ be an $n$-parameter family of $A_\infty$-structures and let $\Delta$ be a cocycle representing an element of $\mathcal{D}_m^l$. Then we can define an $(n+1)$-parameter family of $A_\infty$-structures $\tilde{m}\in W[[t_0,t_1,\cdots,t_n]]$ as a unique formal solution to the differential equation 
\begin{equation}\label{tm}
\tilde{m}_{(0)}=\Delta[\tilde m]
\end{equation}
subject to the initial condition $\tilde{m}|_{t_{0}=0}=m$. Here the new formal parameter $t_{0}$ has degree $1-l$. 
\end{proposition}

\begin{proof}
It is clear that Eq.(\ref{tm}) has a unique formal solution that starts as $$\tilde{m}(t_0)= m+t_0\Delta[m]+o(t^2_0)\,.$$ 
Differentiating now the cochain $\lambda(t_0)=[\tilde{m},\tilde{m}]$ by $t_0$, we get 
$$
\frac{\partial \lambda}{\partial t_0}=2[\tilde{m}_{(0)},\tilde{m}]=[\Delta[\tilde{m}],\tilde{m}]=0\,.
$$
With account of the initial condition $\lambda(0)=[m,m]=0$ this means that $\lambda(t_0)=0$; and hence, $\tilde{m}$ defines an $(n+1)$-parameter family of $A_\infty$-structures.  

\end{proof}

We call the deformations of Proposition \ref{def} the {\it inner deformations} of families of $A_\infty$-algebras.

\begin{remark}\label{VF}
Geometrically, we can think of the r.h.s. of (\ref{tm}) as a vector field $\Delta$ on the infinite-dimensional space $\mathcal{W}$. Then the $A_\infty$-structures form a submanifold $\mathcal{M}\subset \mathcal{W}$ defined by the quadratic equation $[m,m]=0$.  The cocycle condition  $[m,\Delta]=0$ means that the vector field $\Delta$ is tangent to $\mathcal{M}$ and generates a flow  $\Phi^\Delta_{t_0}$ on $\mathcal{W}$, which leaves $\mathcal{M}$ invariant.  Therefore, $\tilde{m}=\Phi^\Delta_{t_0}(m)\subset \mathcal{M}$. Proceeding with this geometrical interpretation, we can consider the commutator $[\![\Delta,\Delta']\!]$ of two vector fields $\Delta$ and $\Delta'$ associated with some elements of $\mathcal{D}_m$. The vector field $[\![\Delta,\Delta']\!]$, being  tangent to $\mathcal{M}$, defines an $M_1$-cocycle.  It would be interesting to study the Lie algebra of vector fields generated by the elements of $\mathcal{D}_m$ in more detail. 

If we now allow some of the parameters $t_i$ to have odd degrees, then an odd vector field $\Delta$ may not be integrable in the sense that $[\![\Delta, \Delta]\!]\neq 0$. In this case Eq. (\ref{tm}) for the flow should be modified as 
$$
\tilde{m}_{(0)}=\Delta[\tilde{m}]-\frac12 t_0[\![\Delta,\Delta]\!][\tilde{m}]\,.
$$
Since $(t_0)^2=0$, the solution is given by $\tilde{m}=m+t_0\Delta[m]$ and it is obvious that $\tilde{m}\circ\tilde{m}=0$.
\end{remark}

\subsection{ Deformation of MC elements} Let $(V,m)$ be a flat $A_\infty$-algebra.  Then, whenever it is defined, the MC equation reads 
\begin{equation}\label{MCEq}
    m(a):=\sum_{n=1}^\infty m_n(a,\ldots, a)=0
\end{equation}
for $|a|=0$. A solution $a\in V^0$ to this equation is called an {MC element} of the $A_\infty$-algebra $(V,m)$  and the  set of all MC elements, called the { MC space}, is denoted by $\mathcal{MC}(V,m)$. 

In order to ensure the convergence of the series (\ref{MCEq}) one or another assumption about $(V,m)$ is needed. For example, one may assume that $m_n=0$ for all $n>p$, so that the series (\ref{MCEq}) is actually finite. This is the case of DGA's. Another possibility is to consider the scalar extension $V\otimes \mathfrak{m}_A$, where $\mathfrak{m}_A$ is the maximal ideal of an Artinian algebra $A$; the multi-linear operations on $V$ extend to those on $V\otimes \mathfrak{m}_A$ by $A$-linearity.  Neither of these approaches, however, is appropriate to our purposes. What suits us is, in a sense,  a combination of both.

\begin{definition}\label{d4}
We say that an $n$-parameter family of $A_\infty$-structures $m\in W[[t_1,\ldots,t_n]]$ is {\it locally finite} if for each 
$k$ there exists a finite $N$ such that
$$
m_{(i_1\cdots i_k)}|_{t=0}\in \bigoplus_{n=0}^N\mathrm{Hom}(T^n(V),V)\,.
$$
\end{definition}
With the supposition of  local finiteness, the MC equation (\ref{MCEq}) for an element $a\in V[[t_1,\ldots,t_n]]$ gives an infinite collection  of  well-defined  equations on the Taylor coefficients of $a$. Furthermore, we have the following statement, whose proof is left to the reader.
\begin{proposition}
Any inner deformation of a locally finite family of $A_\infty$-algebra is locally finite. 
\end{proposition}

The next proposition shows that inner deformations of locally finite families  are always accompanied by deformations of their MC spaces.  

\begin{proposition}\label{p6}
Let $\tilde{m}$ be an inner deformation of a locally finite family of $A_\infty$-structures $m$. Then each MC element for $m$ can be deformed to that for $\tilde{m}$, establishing thus a monomorphism $\mathcal{MC}(V,m)\rightarrow \mathcal{MC}(V,\tilde{m})$. 
\end{proposition}

\begin{proof}
In order to simplify the formulas below we restrict ourselves to inner deformations that are generated by monomials 
$$
\Delta[\tilde{m}]=\tilde{m}_{(i_1)}\cup \tilde{m}_{(i_2)}\cup \cdots \cup \tilde{ m}_{(i_l)}\,.
$$
The generalization  to arbitrary cup polynomials (\ref{pol}) will be obvious. 

Suppose that $a\in \mathcal{MC}(V,\tilde{m})$, then 
\begin{equation}\label{tma}
\tilde{m}(a)=0\,.
\end{equation}
Differentiating this identity by $t_i$, we get
$$
\tilde{m}_{(i)}(a)+\tilde{m}\{D_i\}(a)=0\,.
$$
Here  by $D_i$ we denoted the operator of partial derivative, $D_ia=\partial a/\partial t_i$. 
Therefore, when evaluated on MC elements, the deformation equation (\ref{tm}) can be written as 
\begin{equation}\label{mDs}
\tilde{m}\{D_0\}(a)=-\Delta[\tilde{m}](a)\,.
\end{equation}
More explicitly, the r.h.s. of this equation is defined by  
$$
\begin{array}{c}
\Delta[\tilde{m}](a)=\Delta(\tilde{m}_{(i_1)},\tilde{m}_{(i_2)},\ldots, \tilde{m}_{(i_l)} )(a)\\[3mm]
=M_2(\ldots( M_2(M_2(M_2(\tilde{m}_{(i_1)},\tilde{m}_{(i_2)}),\tilde{m}_{(i_3)}),\ldots, \tilde{m}_{(i_l)})(a)\,.
\end{array}
$$
Here we used the definition of the cup product (\ref{cup}). 

Again, when evaluated on MC elements,  the expression $\tilde{m}\{D_i\}(a)$ can be replaced with $M_1(D_i)(a)$. This allows us to write 
$$
\Delta[\tilde{m}](a)=-\Delta(M_1(D_{i_1}),\tilde{m}_{(i_2)},\ldots, \tilde{m}_{(i_l)} )(a)\,.
$$
Since $M_1(\tilde{m}_{(i)})=0$, the repeated use of Rel. (\ref{M12}) allows us to rewrite the last expression as  
$$
\Delta[\tilde{m}](a)=(-1)^{l} M_1(\Delta(D_{i_1},\tilde{m}_{(i_2)},\ldots, \tilde{m}_{(i_l)} ))(a)
$$
or, equivalently, 
$$
\Delta[\tilde{m}](a)=(-1)^{l} m\{\Delta(D_{i_1},\tilde{m}_{(i_2)},\ldots, \tilde{m}_{(i_l)} )\}(a)\,.
$$
Thus, Eq. (\ref{mDs}) takes the form 
$$
m\{D_0\}(a)=-(-1)^l m\{\Delta(D_{i_1},\tilde{m}_{(i_2)},\ldots, \tilde{m}_{(i_l)} )\}(a)\,.
$$
We will definitely  satisfy this equation if require that 
\begin{equation}\label{Da}
D_0a=-(-1)^l \Delta(D_{i_1},\tilde{m}_{(i_2)},\ldots, \tilde{m}_{(i_l)})(a)\,.
\end{equation}
This gives a differential equation for $a\in V[[t_0,t_1,\ldots,t_n]]$ w.r.t. the formal `evolution parameter' $t_0$ of degree $1-|\Delta|$.

Now we can evaluate $\tilde{m}$ on a formal solution to Eq. (\ref{Da}). It follows from the course of the proof above  that the partial derivative $D_0(\tilde{m}(a))$ depends on $\tilde{m}(a)$ linearly thereby vanishes on (\ref{tma}). This means that the vector $\tilde{m}(a)\in V[[t_0,t_1,\ldots,t_n]]$ is zero  whenever it vanishes at $t_0=0$. But the last condition is just the definition of an MC element $a|_{t_0=0}\in \mathcal{MC}(V,m)$.  
\end{proof}

\subsection{Minimal deformations of DGA's}\label{MinDef} We now apply the above machinery of inner deformations to the case of DGA's. Recall that a DGA $\mathcal{A}$ is given by a triple $(V,\partial,\cdot)$, where $V=\bigoplus V^l$ is a graded vector space endowed with an associative dot product and a differential $\partial: V^l\rightarrow V^{l-1}$. 

\begin{remark} Here we equip a DGA with a differential of degree $-1$. From the perspective of $A_\infty$-algebras, it is more natural to consider differentials of degree $1$. As was discussed in Sec. 2, a DGA structure on $V$ can then be interpreted as a `degenerate' $A_\infty$-structure on $V[1]$ involving only a linear map $m_1$  and a bilinear map $m_2$, both of degree $1$. Actually, there is not much difference between the two definitions as one can always relate them by the degree reversion functor $\iota$. By definition, $\iota V$ is a graded vector space with $(\iota V)^l=V^{-l}$.   Clearly, the $k$-linear map $V\rightarrow \iota V$ respects the product while reverting the degree of the differential. 
\end{remark}

Let $\mathcal{A}_t$ be a one-parameter deformation of $\mathcal{A}$, with $t$ being a formal parameter of degree zero.  In order to make the DGA $\mathcal{A}_t$ into a family of $A_\infty$-algebras, we define the tensor product algebra $\mathcal{A}_t\otimes k[[u]]$, where $u$ is an auxiliary formal variable  of degree $2$. Here we consider $k[[u]]$ as a DGA with trivial differential.  Multiplying now the differential $\partial$ in $\mathcal{A}_t$ by $u$ yields the differential $d=u\partial$ in $\mathcal{A}_t\otimes k[[u]]$ of degree $1$. 
This allows us to treat the DGA $\mathcal{A}_t\otimes k[[u]]$  as a $2$-parameter family of $A_\infty$-algebras with $m_1=d$ and $m_2$ defined by (\ref{bbb}).  On the other hand, given  a two-parameter family of $A_\infty$-structures $m$ with the parameters $t$ and $u$ of degrees $0$ and $2$, respectively, we can define the sequence of cocycles 
\begin{equation}\label{Dn}
\Delta_n=m_{(t)}\cup \underbrace{ m_{(u)}\cup  m_{(u)}\cup \cdots \cup m_{(u)} }_n\,,\qquad n=0,1,2,\ldots\,.
\end{equation}
Here the subscripts $t$ and $u$ stand for the partial derivatives of $m$ w.r.t. $t$ and $u$. 
Keeping in mind that the cup product has degree $1$ while $|m_{(u)}|=-1$, we conclude that $|\Delta_n|=1$ for all $n$.  By Proposition \ref{def}, each cocycle $\Delta_n$ gives rise to a formal deformation of $m$ with a new deformation parameter $s$ of degree zero. The deformed $A_\infty$-structure $\tilde{m}$ is defined by the differential equation 
\begin{equation}\label{tilde-m}
\tilde{m}_{(s)}=\Delta_n[\tilde{m}]
\end{equation}
with the initial condition $\tilde{m}|_{s=0}=m$. The parameter $u$ plays an auxiliary role in our construction. Setting $u=0$, we finally get, for each $n$,  a family $\bar{m}=\tilde{m}|_{u=0}$ of $A_\infty$-structures parameterized by  $t$ and $s$; both the parameters are of degree zero. 
By construction, $\bar{m}$ starts with $m_2$ and the first-order deformation in $s$ is given by
$$
\bar{m}^{(1)}(a_1,a_2,\ldots,a_{n+2})=(a_1\cdot a_2)'\cdot\partial(a_3)\cdot \partial(a_4)\cdots \partial(a_{n+2})\,,
$$
where the prime denotes the partial derivative of the dot product in $\mathcal{A}_t$ by $t$. 
Evaluating $\bar{m}\circ \bar{m}=0$ at the first order in $s$, we get  
$$
\begin{array}{c}
[m_2, \bar{m}^{(1)} ](a_0,a_1,\ldots,a_{n+2})=-a_0 \cdot \bar{ m}^{(1)}(a_1,a_2,\ldots,a_{n+2})\\[5mm]
\displaystyle -\sum_{k=0}^{n-1}(-1)^{|a_0|+\cdots+|a_k|}\bar{m}^{(1)}(a_0,\ldots, a_{k-1},a_k\cdot a_{k+1},a_{k+2},\ldots, a_{n+2})\\[5mm]
+(-1)^{|a_0|+\cdots +|a_{n+1}|} \bar{m}^{(1)}(a_0,a_1,\ldots, a_{n+1})\cdot a_{n+2}=0\,.
\end{array}
$$
Therefore, $\bar{m}^{(1)}$ is a Hochschild cocycle of the algebra $\mathcal{A}_t$ representing an element of $HH^{n+2}(\mathcal{A}_t,\mathcal{A}_t)$. If the cocycle $m^{(1)}$ is nontrivial, then it defines a nontrivial deformation of the algebra $\mathcal{A}_t$ in the category of $A_\infty$-algebras.  Notice that  the resulting $A_\infty$-structure $\bar{m}$ is minimal as, by construction, $\bar{m}_1=0$.  For this reason we refer to $\bar{m}$ as a {\it minimal deformation} of the DGA structure $m$. In such a way we arrive at the first statement of Theorem \ref{Th1}. 

In the special case that  the differential $\partial$ does not depend on $t$, the r.h.s. of Eq. (\ref{tilde-m}) is independent of $u$, so that the whole dependence of $\tilde{m}$ of $u$ is concentrated in the first structure map $\tilde{m}_1=u\partial$. This means that all the structure maps constituting $\tilde{m}$ or $\bar{m}$ are differentiated by $\partial$.

Being determined only by the first and second structure maps, the $A_\infty$-algebra  $\mathcal{A}_t\otimes k[[u]]$ is evidently locally finite  in the sense of Definition \ref{d4} and so is its minimal deformation defined by $\bar{m}\in W[[t,s]]$.  At $s=0$, the $A_\infty$-structure $\bar{m}$ reduces to the product in $\mathcal{A}_t$ and the MC equation takes the form\footnote{Notice that the differential $\partial$ does not contribute to the MC equation, contrary to what one might expect. From the viewpoint of the DGA the element  $a$ has degree $1$, so that $\partial a$ is of degree $0$, and not $2$.} $a\cdot a=0$.
Applying Proposition \ref{p6} to a solution $a$ yields then an MC element $\bar{a}=a+\sum_{k>0}a_{k}s^k$ for the minimal $A_\infty$-algebra $(V, \bar{m})$, that is, $\bar{m}(\bar{a})=0$. This proves the rest part of Theorem \ref{Th1}.

As a final remark we note that the above construction of minimal deformations carries over verbatim to the case of smooth (i.e., not formal) families of DGA's $\mathcal{A}_t$. 
An interesting example of a smooth family of algebras is considered below.

\subsection{Example.} Let us illustrate the construction of the previous subsection by the example of the polynomial Weyl algebra $A_m[t]$. As a vector space $A_m[t]$ coincides with the space $k[t,x^1,\ldots, x^{2m}]$ of polynomials in $2m+1$ variables.  Multiplication in $A_m[t]$ is given by the $\ast$-product
\begin{equation}\label{WM}
a\ast b=a\cdot b +\sum_{k=1}^\infty {t^k} (a\stackrel{_k}{\ast} b)\,,
\end{equation}
where 
$$
a\stackrel{_k}{\ast} b= \frac1{k!}\omega^{i_1j_1}\cdots \omega^{i_kj_k}\frac{\partial^k a}{\partial x^{i_1}\cdots \partial x^{i_k}}\frac{\partial^k b}{\partial x^{j_1}\cdots \partial x^{j_k}}
$$
and $\omega^{ij}$ is a skew-symmetric, non-degenerate matrix with entries in $k$. The $\ast$-product is known to be associative but non-commutative. 
Clearly, one may regard $A_m[t]$ as a one-parameter deformation of the usual polynomial algebra $k[x^1,\ldots,x^{2m}]$ with the commutative dot product.

As was explained in the Introduction, we can turn the polynomial Weyl algebra  into a family of DGA's $\mathcal{A}_t$ simply treating $A_m[t]$ as a bimodule over itself. 
The family $\mathcal{A}_t$ is concentrated in degrees $0$ and $1$, so that the underlying $k$-vector space is $V=V^0\oplus V^1$ with  $V^0=A_m[t]=V^1$.  Multiplication is defined by the rule (\ref{Ext}) and the differential $\partial$ is completely specified by declaring $\partial : V^1 \rightarrow V^0$ to be the identity homomorphism of $A_m[t]$ onto itself. 

Following prescriptions of Sec. \ref{MinDef}, we can now produce a two-parameter family of $A_\infty$-structures $\bar{m}$ generated, for example, by the cocycle $\Delta_1$ of (\ref{Dn}).  The family is parameterized by the initial parameter $t$ and a 
new deformation parameter $s$ of degree zero. The explicit expression for $\bar{m}$ resulting from the deformation equation (\ref{tilde-m}) appears to be rather complicated.  The situation is slightly simplified if we put $t=0$. This corresponds to a deformation of the polynomial algebra $k[x^1,\ldots,x^{2m}]$, considered as a bimodule over itself, in the category of minimal $A_\infty$-algebras. 

A direct, albeit tedious, calculation shows that the only nonzero maps $m_n\in \mathrm{Hom}(T^n(V[1]),V[1])$ 
constituting the $A_\infty$-structure $m=\bar{m}|_{t=0}$ are given by 
\begin{equation}\label{mmm}
\begin{array}{rl}
    m_n(a,b,u_3,\ldots,u_{n})&=s^{n-2}f_n(a,b,u_3,\ldots,u_{n-1})\cdot u_{n}\,,\\[3mm]
    m_n(a,u_2,\ldots,u_{n})&=s^{n-2}f_n(a,u_2,\ldots,u_{n-1})\cdot u_{n}\,,\\[3mm]
    m_n(u_2,b,u_3,\ldots,u_{n})&=-s^{n-2}f_n(u_2,b,\ldots,u_{n-1})\cdot u_{n}\,,
\end{array}
\end{equation}
where $a,b\in V^0$, $u_2,\ldots, u_n\in V^1$,  and  the $f_n$'s are defined by  
$$
f_{n+1}(a_1,a_2,\ldots, a_n)
$$
$$
=\sum a_1\stackrel{_{k_1}}{\ast} \underbrace{a_2\cdot a_3\cdots a_{l_1+1}}_{l_1}\stackrel{_{k_2}}{\ast}\underbrace{a_{l_1+2}\cdots a_{l_1+l_2+2}}_{l_2}\stackrel{_{k_3}}{\ast}\cdots \stackrel{_{k_p}}{\ast}\underbrace{a_{n-l_p+1}\cdots a_{n-1}\cdot a_{n}}_{l_p}
$$
for all $a_i\in k[x^1,\ldots,x^{2m}]$. Here, to save space, we omit parentheses specifying the order of multiplication; it is understood that all the multiplication operations are performed from left to right and 
summation runs over  all $k$'s and $l$'s obeying the (in)equalities 
$$
    \sum_{j=1}^{p} l_j=\sum_{j=1}^p k_j=n-2\,,\qquad  l_j\geq 1\,,\qquad k_j\geq 1\,, \qquad p\geq 0\,,
$$
$$
    l_p\geq k_p\,,\qquad 
    l_{p-1}+l_p\geq k_{p-1}+k_p\,,\quad 
    \ldots\,,\quad
    l_2+\cdots +l_p\geq k_2+\cdots +k_p\,.
$$

In particular, for $n=2$ (which means $p=0$) we recover the original bimodule structure for the polynomial algebra $k[x^1,\ldots,x^{2m}]$: 
$$
m_2(a,b)=a\cdot b \,,\qquad m_2(a,u)=a\cdot u\,,\qquad m_2(u,b)=-u\cdot b\,,
$$
and the first-order deformation is given by 
$$
\begin{array}{l}
m_3(a,b,u)=s(a\stackrel{_1}{\ast}b)\cdot u\,,\\[3mm] m_3(a,u_1,u_2)=s(a\stackrel{_1}{\ast}u_1)\cdot u_2\,,\\[3mm] 
m_3(u_1,b,u_2)=-s(u_1\stackrel{_1}{\ast}b)\cdot u_2\,.
\end{array}
$$
It is not hard to see that this deformation is nontrivial. Finally, for all $n$
$$
m_{n+2}(a,b,1,\ldots,1)=s^na\stackrel{_n}{\ast} b\,,
$$
which allows us to regard (\ref{mmm}) as a certain $A_\infty$ generalization of the Weyl--Moyal $\ast$-product (\ref{WM}) to the case of a `non-constant deformation parameter' $u$.

\end{document}